\documentclass[11pt,bezier]{article}
\usepackage{times}
\usepackage{booktabs}
\usepackage{pifont}
\usepackage{floatrow}
\usepackage{caption}
\usepackage{mathrsfs}
\usepackage{amsmath}
\usepackage{amsfonts,amsthm,amssymb,mathrsfs,bbding}
\usepackage{txfonts}
\usepackage{graphics,multicol}
\usepackage{graphicx}
\usepackage{color}
\usepackage{longtable}

\usepackage[colorlinks,
            linkcolor=black,
            anchorcolor=black,
            citecolor=black]{hyperref}
\usepackage{url}
\usepackage{caption}
\usepackage{cite}
\usepackage{latexsym,bm}
\usepackage{indentfirst}
\usepackage{color}
\usepackage[T1]{fontenc}
\usepackage[utf8]{inputenc}
\usepackage{authblk}
\evensidemargin=\oddsidemargin\topmargin=-1.5cm

\newtheorem{theorem}{Theorem}[section]

\newtheorem{lemma}{Lemma}[section]
\newtheorem{cor}{Corollary}[section]

\theoremstyle{definition}
\newtheorem{remark}{Remark}

\newtheorem{defi}{Definition}[section]

\theoremstyle{definition}

\addtocounter{section}{0}

\newtheoremstyle{nonumber}%
  {3pt}{3pt}%
  {\itshape}{}%
  {\bfseries}{.}
  { }%
  {\thmname{#1}\thmnote{ #3}}
\theoremstyle{nonumber}
\newtheorem*{theorem*}{Theorem}

\begin{document}

\title{\bf\Large Edge-Connectivity versus Lin--Lu--Yau Curvature\thanks{Supported by NSFC (Nos. 12131013, 12471334).}}

\date{}

\author{Ziming Zhou$^{1,2}$, Shenggui Zhang$^{1,2,}$\thanks{Corresponding author. E-mail addresses: zhouziming@mail.nwpu.edu.cn(Z. Zhou), sgzhang@nwpu.edu.cn (S. Zhang).}\\
\small $^{1}$School of Mathematics and Statistics, \\
\small Northwestern Polytechnical University, Xi'an, Shaanxi 710129, China\\
\small $^{2}$Xi'an-Budapest Joint Research Center for Combinatorics, \\
\small Northwestern Polytechnical University, Xi'an, Shaanxi 710129, China\\}

\maketitle

\begin{abstract}
We systematically explore the relationship between edge-connectivity and Lin--Lu--Yau curvature for graphs. 
The intuition is that a connected graph with large Lin--Lu--Yau curvature should also have large edge-connectivity, and vice versa under suitable conditions. 
We prove that the edge-connectivity of a connected graph is bounded lower by the product of its minimum degree, its Lin--Lu--Yau curvature and the constant $\frac{3}{2}$. 
We also provide two lower bounds on the edge-connectivity of a graph that guarantee positive Lin--Lu--Yau curvature, 
and discuss the sharpness of these bounds.
  
\medskip

\begin{flushleft}
\end{flushleft}

\end{abstract}

\section{Introduction}

A graph is \textit{connected} if, for every partition of its vertex set into two nonempty sets $X$ and $Y$, there is an edge with one end in $X$ and one end in $Y$; otherwise it is \textit{disconnected}.
Connectivity is one of the fundamental concepts of graph theory.
There are two parameters of a graph, its connectivity and edge-connectivity, which measure the extent to which it is connected\cite{Bondy2008, Diestel2017}.

The \textit{connectivity} $\kappa(G)$ of a non-complete graph $G$ is defined as the minimum number of vertices whose removal disconnects the graph.
The \textit{edge-connectivity} $\kappa'(G)$ of a graph $G$ with at least two vertices is defined as the minimum number of edges whose removal disconnects the graph.
By convention, the connectivity of a complete graph $K_n$ with $n$ vertices is defined to be $n-1$, and the edge-connectivity of a graph with only a single vertex is defined to be $0$. 

A classical bound on connectivity due to Whitney \cite{Whitney1932} reads as
$$
\kappa(G) \leq \kappa'(G) \leq \delta(G),
$$
where $\delta(G)$ denotes the minimum vertex degree of $G$. 
Both inequalities are sharp. For instance, the identity $\kappa(G)=\kappa'(G)=\delta(G)$ holds for any finite connected edge-transitive graph \cite{Mader1971, Watkins1970}, for any finite distance-regular graph \cite{BrouwerHaemers2005, BrouwerKoolen2009, BrouwerMenser1985}, and for any finite $(0,2)$-graph \cite{BrouwerMulder1997}. 
In general, these three quantities $\kappa(G)$, $\kappa'(G)$, and $\delta(G)$ can differ by arbitrarily large amounts. Indeed, for any integers $0 < \kappa \leq \kappa' \leq \delta$, there exists a graph $G$ with $\kappa(G)=\kappa$, $\kappa'(G)=\kappa'$, and $\delta(G)=\delta$.

For over a century, Ricci curvature has been a fundamental concept in the study of smooth Riemannian manifolds, governing volume growth and geodesic behavior. 
Because of its importance, there has been considerable interest in generalizing Ricci curvature to discrete structures in recent decades. 
In 2009, Ollivier \cite{Ollivier} introduced a notion of coarse Ricci curvature for Markov chains on metric spaces via optimal transport. This notion is particularly well suited for locally finite graphs. Lin, Lu, and Yau \cite{LLY} modified Ollivier's definition to introduce the Lin--Lu--Yau curvature on graphs in 2011. This refinement is obtained by differentiating the Ollivier-Ricci curvature with respect to the idleness parameter. The Lin--Lu--Yau curvature detects local structural information of a graph and has been widely studied in recent years. 

The Lin--Lu--Yau curvature is defined for any two distinct vertices $x,y$ of a graph $G$. When $x$ and $y$ are adjacent, we call $\kappa_{\mathrm{LLY}}(x,y)$ the Lin--Lu--Yau curvature of the edge $xy$.
The precise definition is formulated via the Wasserstein distance between probability measures; see Subsection \ref{Lin--Lu--Yau curvature}. In Figure \ref{example}, the number on each edge is the Lin--Lu--Yau curvature of that edge, computed by the graph curvature calculator\cite{Cushing2022}.

\begin{figure}[htbp!]
\centering
\includegraphics[width=14cm,page={1}]{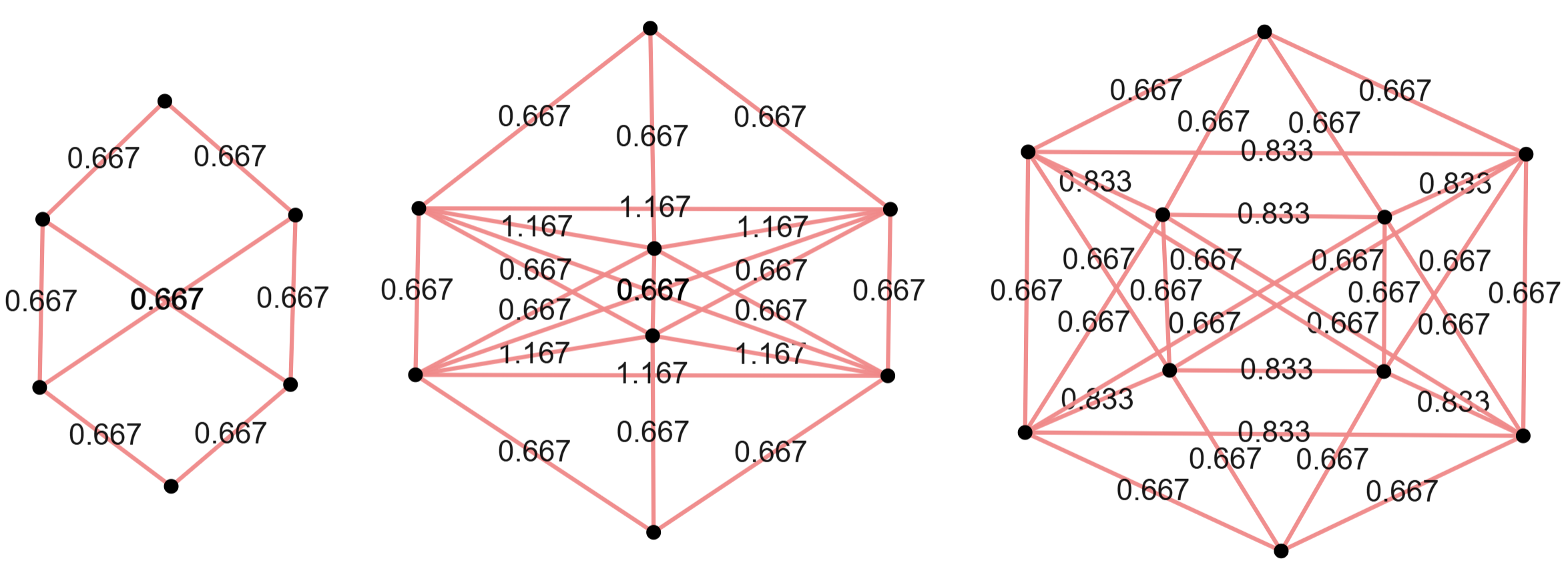}
\caption{Three graphs labeled with the Lin--Lu--Yau curvature of each edge. The figure is produced by the interactive app “Graph curvature calculator” at https://www.mas.ncl.ac.uk/graph-curvature/.}
\label{example}
\end{figure}

For two vertices $x, y$ in $G$, we denote by $d(x, y)$ the distance between $x$ and $y$. For any positive integer $i$, let us denote the lower Lin--Lu--Yau curvature bound at scale $i$ by
$$
\kappa_{\mathrm{LLY}}^{(i)}(G) := \inf_{\substack{x, y \in V \\ d(x, y) = i}} \kappa_{\mathrm{LLY}}(x, y).
$$
When $i = 1$, we simply write $\kappa_{\mathrm{LLY}}(G)$ instead of $\kappa_{\mathrm{LLY}}^{(1)}(G)$.

Chen, Liu and You\cite{Liu} showed that positive Lin--Lu--Yau curvature implies large connectivity. 
They proved that $\kappa(G) \geq \delta(G)  \kappa_{\mathrm{LLY}}^{(2)}(G)$ holds for every non-complete connected graph $G$ with minimum degree $\delta(G)$ and connectivity $\kappa(G)$.
They further characterized the structure of an infinite graph family for which equality is attained in the inequality. 
Actually, $$\kappa'(G) \geq \delta(G)  \kappa_{\mathrm{LLY}}^{(2)}(G)$$ holds for every non-complete connected graph $G$ with diameter $\mathrm{diam}(G)=2$ 
since $\kappa'(G)\geq\kappa(G)$.

For any graph $G$ obtained by deleting a non-empty matching from a complete graph $K_n$ with $n \geq 3$ vertices, the equality $\kappa'(G) = \delta(G)  \kappa_{\mathrm{LLY}}(G)$ holds, and hence equality $\kappa'(G) = \delta(G)  \kappa_{\mathrm{LLY}}^{(2)}(G)$ also holds since $\kappa_{\mathrm{LLY}}^{(2)}(G)  \geq \kappa_{\mathrm{LLY}}(G)$ by the triangle inequality of the Wasserstein distance; see \cite[Lemma 2.3]{LLY} or Lemma \ref{triangle inequality} below.
Indeed, it was proved independently in \cite[Theorem 1.3]{Cushing} and \cite[Theorem 1]{Huang} that $\kappa_{\mathrm{LLY}}(G) = 1$ for such a graph. Therefore, $\kappa'(G) \geq \delta(G)$, which forces equality.
Such a graph satisfies $\mathrm{diam}(G)=2$, whereas for graphs with $\mathrm{diam}(G)\geq3$, a better bound on edge-connectivity is available.
Our first result shows that positive Lin--Lu--Yau curvature implies large edge-connectivity.


\begin{theorem}\label{diam3}
Let $G$ be a non-complete connected graph with diameter $\mathrm{diam}(G)\geq3$, minimum degree $\delta(G)$ and edge-connectivity $\kappa'(G)$. Then,
$$
\kappa'(G) \geq \frac{3}{2}\delta(G)  \kappa_{\mathrm{LLY}}^{(3)}(G).
$$
\end{theorem}

Recall that $\kappa_{\mathrm{LLY}}^{(3)}(G)  \geq \kappa_{\mathrm{LLY}}(G)$ by the triangle inequality for the Wasserstein distance. Thus, the corollary below follows immediately.

\begin{cor}
Let $G$ be a non-complete connected graph with minimum degree $\delta(G)$ and edge-connectivity $\kappa'(G)$. Then,
$$\kappa'(G) \geq \frac{3}{2}\delta(G)  \kappa_{\mathrm{LLY}}(G).$$
\end{cor}

For the three graphs in Figure \ref{example}, the equality $\kappa'(G) = \frac{3}{2}\delta(G)  \kappa_{\mathrm{LLY}}(G)$ holds, and hence equality also holds in Theorem \ref{diam3}.
Moreover, let $G_{2n}\cong K_n \cup K_n +M_2$ be the union of two copies of the complete graph $K_n$ ($n\geq3$) with a matching of size 2 between them. Equality holds in Theorem \ref{diam3} for $G_{2n}$; see Lemma \ref{distance3} below.  
In fact, the graphs $G_{2n}$ form an infinite family for which Theorem \ref{diam3} is sharp.

In Riemannian geometry, spaces with positive Ricci curvature enjoy numerous properties, many of which have natural probabilistic interpretations.
We refer to the comprehensive survey \cite{Lott} and the references therein for a discussion of the geometric interest of lower bounds on Ricci curvature and the need for generalized notions of positive Ricci curvature for metric measure spaces.

Hehl \cite{Hehlmin} established a lower bound on the minimum degree that guarantees non-negative Lin--Lu--Yau curvature. He proved that if $\delta(G)\geq \frac{2n}{3}-1$, then $\kappa_{\mathrm{LLY}}(G)\geq0$. Moreover, the bound $\frac{2n}{3}-1$ is optimal.
There is also a lower bound on the connectivity that guarantees positive Lin--Lu--Yau curvature: namely, if $\kappa(G)\geq \frac{n-1}{2}$, then $\kappa_{\mathrm{LLY}}(G)>0$, and the bound $\frac{n-1}{2}$ is optimal; see \cite[Theorem 1.7 and Theorem 1.8]{Liu}.

We derive sufficient conditions on edge-connectivity for positive Lin--Lu--Yau curvature.
The following results show that sufficiently large edge-connectivity implies large Lin--Lu--Yau curvature.

\begin{theorem}\label{n bound}
Let $G$ be a connected graph with $n$ vertices and edge-connectivity $\kappa'(G)$. If 
$$
\kappa'(G)\ge \frac{2n}{3}-1,
$$ 
then $\kappa_{\mathrm{LLY}}(x,y)>0$ for every edge $xy$ of $G$.
\end{theorem}

After deriving this lower bound on the edge-connectivity of a graph that guarantees positive Lin--Lu--Yau curvature, we proceed to discuss the sharpness of the bound.
We notice that the lower bound $\frac{2n}{3}-1$ is optimal. There exist infinitely many graphs with edge-connectivity $\frac{2n}{3}-2$, each containing an edge of negative Lin--Lu--Yau curvature. Precisely, we have the following result.

\begin{theorem}\label{n bound sharpness}
Let $n \in \mathbb{N}$ such that $2n/3 \in 2\mathbb{N}$ and $2n/3 \geq 6$. Then there exist infinitely many graphs with $n$ vertices and edge-connectivity 
$$
\kappa'(G) = \frac{2n}{3} - 2
$$ 
that contain an edge $xy$ with $\kappa_{\mathrm{LLY}}(x, y) < 0$.
\end{theorem}

We derive a sufficient condition, depending only on the number of edges, for positive Lin--Lu--Yau curvature.

\begin{theorem}\label{m bound}
Let $G$ be a connected graph with $m$ edges and edge-connectivity $\kappa'(G)$. If $$
\kappa'(G)\ge \left\lceil\frac{\sqrt{1+\frac{16}{3}m}-1}{2}
\right\rceil,
$$ 
then $\kappa_{\mathrm{LLY}}(x,y)>0$ for every edge $xy$ of $G$.
\end{theorem}

We also discuss the sharpness of the bound.

\begin{theorem}\label{m bound sharpness}
Let $m \in \mathbb{N}$ such that $m \geq 11$. Then there exist infinitely many graphs with $m$ edges and edge-connectivity 
$$
\kappa'(G)=\left\lceil\frac{\sqrt{1+\frac{16}{3}m}-1}{2}\right\rceil-1
$$ 
that contain an edge $xy$ with $\kappa_{\mathrm{LLY}}(x, y) = 0$.
\end{theorem}

\section{Preliminaries}

In this section, we revisit key concepts from Graph Theory and Optimal Transport Theory.
Next, we collect useful results on Ollivier’s discrete Ricci curvature for graphs and its modification by Lin, Lu, and Yau.

\subsection{Terminology and  notations}
Throughout the paper, we use the following notations. 
Let $G = (V, E)$ be an undirected simple locally finite graph and $V, E$ respectively denote the set of vertices and edges of $G$.

For $x \in V$ and $r \in \mathbb{N}$ we define the \textit{$r$-sphere centered at $x$} as $S_r(x) = \{y \in V : d(x,y) = r\}$ and the \textit{$r$-ball centered at $x$} as $B_r(x) = \{y \in V : d(x,y) \leq r\}$. Denote by $d_x = |S_1(x)|$ the \textit{degree} of a vertex $x \in V$. The \textit{minimum degree} of a graph $G$ is denoted by $\delta(G) = \min_{x \in V} d_x$. We call a graph $G$ \textit{locally finite} if every vertex has finite degree. The graph is said to be \textit{$d$-regular} if every vertex has the same degree $d$.
Given an edge $xy$, we denote the set of common neighbors of $x$ and $y$ by $N_{xy} = S_1(x) \cap S_1(y)$.

The \textit{distance} between two vertices $x, y$ is denoted by $d(x, y)$. 
The \textit{diameter} of $G$ is denoted by $\mathrm{diam}(G) = \max_{x,y \in V} d(x,y)$.
For two vertex sets $X$ and $Y$, the distance between $X$ and $Y$ is denoted by
$$d(X, Y) := \min \{ d(x, y) \mid x \in X, y \in Y \}.$$
When $X = \{x\}$, we simply write $d(x, Y)$ instead of $d(\{x\}, Y)$. 
For two disjoint vertex sets $X$ and $Y$, we write $E(X, Y)$ for the set of edges with one end in $X$ and the other in $Y$.


\subsection{Ollivier-Ricci and Lin--Lu--Yau curvature}\label{Lin--Lu--Yau curvature}

We first recall the definition of Wasserstein distance, which is a metric on the space of probability measures and plays a fundamental role in optimal transport theory.
For a graph $G = (V, E)$, a \textit{probability measure} on $G$ is a function $\mu : V \to [0, 1]$ such that $\sum_{x \in V} \mu(x) = 1$.

\begin{defi}[Wasserstein distance] 
Let $G = (V, E)$ be a locally finite graph. Let $\mu_1, \mu_2$ be two probability measures on $G$. The \textit{Wasserstein distance} $W(\mu_1, \mu_2)$ between $\mu_1$ and $\mu_2$ is defined as
$$
W(\mu_1, \mu_2) = \inf_{\pi} \sum_{x, y \in V} d(x, y) \pi(x, y),
$$
where the infimum is taken over all the maps $\pi : V \times V \to [0, 1]$ satisfying
$$
\mu_1(x) = \sum_{y \in V} \pi(x, y), \quad \forall x \in V
$$
and
$$
\mu_2(y) = \sum_{x \in V} \pi(x, y), \quad \forall y \in V.
$$
\end{defi}

Such a map is called a \textit{transport plan} from $\mu_1$ to $\mu_2$.
Imagine two distributions, $\mu_1$ and $\mu_2$, as piles of earth. A transport plan $\pi$ transforms one pile of earth into another. The minimal effort which is required for such a transformation is measured by the Wasserstein distance. If the infimum is attained, we call $\pi$ an \textit{optimal transport plan} transporting $\mu_1$ to $\mu_2$.

The Kantorovich duality tells the following identity about Wasserstein distance (see, e.g., \cite[Theorem 1.14]{Villani}):
$$
W(\mu_1, \mu_2) = \sup_f \sum_{x \in V} f(x) (\mu_1(x) - \mu_2(x)),
$$
where the supremum is taken over all 1-Lipschitz functions $f$, which means 
$$
|f(x) - f(y)| \leq d(x, y), \quad \forall x, y \in V.
$$

In optimal transport theory, it is well known that no mass needs to be transported when it is already shared between the two probability measures.

\begin{lemma}\label{simple transport plan}   
Let $G = (V, E)$ be a locally finite graph. Let $\mu_1, \mu_2$ be two probability measures on $V$. Then there exists an optimal transport plan $\pi$ transporting $\mu_1$ to $\mu_2$ satisfying
$$
\pi(x, x) = \min\{\mu_1(x), \mu_2(x)\}
$$
for all $x \in V$.
\end{lemma}

To each vertex $x \in V$ and idleness parameter $\alpha \in [0, 1]$, we associate a probability measure $\mu_x^\alpha$ defined by
$$
\mu_x^\alpha(v) =
\begin{cases} 
\alpha, & \text{if } v = x; \\ 
\frac{1-\alpha}{d_x}, & \text{if } xv \in E; \\ 
0, & \text{otherwise}.
\end{cases}
$$

Then, the Ollivier-Ricci curvature is defined as follows.

\begin{defi}[Ollivier-Ricci curvature]
Let $G = (V, E)$ be a locally finite graph. For any two distinct vertices $x$ and $y$, the \textit{$\alpha$-Ollivier curvature} $\kappa_\alpha(x, y)$, $\alpha \in [0, 1]$, is defined as
$$
\kappa_\alpha(x, y) = 1 - \frac{W(\mu_x^\alpha, \mu_y^\alpha)}{d(x, y)}.
$$
The parameter $\alpha$ is called the \textit{idleness}.
\end{defi}

Following \cite{Eidi}, we present a more intuitive formulation of the Ollivier-Ricci curvature. To this end, denote by $\nu_i^\alpha$ the mass transported with distance $i$ under an optimal transport plan transporting $\mu_x^\alpha$ to $\mu_y^\alpha$. Then
$$
\sum_{i=0}^{3} \nu_i^\alpha = 1 \quad \text{and} \quad W(\mu_x^\alpha, \mu_y^\alpha) = \sum_{i=1}^{3} i \nu_i^\alpha.
$$
Hence, for any edge $xy\in E$, we conclude
$$\kappa_\alpha(x, y) = \nu_0^\alpha - \nu_2^\alpha - 2 \nu_3^\alpha.$$
This formula plays a crucial role in proving our results.

Recall that the $\alpha$-Ollivier-Ricci curvature depends on the idleness parameter. Ollivier initially considered the cases $\alpha = 0$ and $\alpha = \frac{1}{2}$, but its behavior across the full range of $\alpha$ remained an open question. Lin, Lu, and Yau \cite{LLY} first established that $\kappa_\alpha$ is concave in $\alpha \in [0, 1]$. Using $\kappa_1(x, y) = 0$, one obtains that the function $h(\alpha) = \frac{\kappa_\alpha(x, y)}{1-\alpha}$ is increasing on the interval $[0, 1)$. By further proving that $h(\alpha)$ is bounded, they concluded that the limit $\lim_{\alpha \to 1} h(\alpha)$ exists. Building on this result, Lin, Lu, and Yau introduced a modified version of the Ollivier-Ricci curvature, obtained by differentiating $\kappa_\alpha$ with respect to the idleness parameter $\alpha$.

\begin{defi}[Lin--Lu--Yau curvature]
Let $G = (V, E)$ be a locally finite graph. For any two distinct vertices $x$ and $y$, the \textit{Lin--Lu--Yau curvature} is defined as
$$
\kappa_{\mathrm{LLY}}(x, y) = \lim_{\alpha \to 1} \frac{\kappa_\alpha(x, y)}{1-\alpha}.
$$
\end{defi}

In \cite{Bourne}, Bourne et al. conducted a more detailed analysis of how $\kappa_\alpha$ depends on the idleness parameter $\alpha$, by introducing the Ollivier-Ricci idleness function $\alpha \mapsto \kappa_\alpha(x, y)$. They demonstrated that the idleness function is piecewise linear with at most three linear parts. Additionally, they determined the length of the final linear part.

\begin{theorem}[\cite{Bourne}, Theorem 4.4]
Let $G = (V, E)$ be a locally finite graph and let $x, y \in V$, $xy \in E$ with $d_x \geq d_y$. Then $\alpha \mapsto \kappa_\alpha(x, y)$ is linear over $\left[ \frac{1}{d_x+1}, 1 \right]$.
\end{theorem}

As a direct consequence of the mean value theorem, they derived the following relation between the $\alpha$-Ollivier-Ricci curvature and its modification by Lin, Lu, and Yau.

\begin{theorem}\label{relation}
Let $G = (V, E)$ be a locally finite graph and let $x, y \in V$, $xy \in E$ with $d_x \geq d_y$. Then
$$
\kappa_\alpha(x, y) = (1-\alpha)\kappa_{\mathrm{LLY}}(x, y)
$$
for $\alpha \in \left[ \frac{1}{d_x+1}, 1 \right]$.
\end{theorem}

Loisel and Romon had previously established this relation in \cite{Loisel} for $\alpha \geq \frac{1}{2}$ in 2014.
Consequently, for large values of $\alpha$, the Lin--Lu--Yau curvature aligns with the $\alpha$-Ollivier-Ricci curvature up to a scaling factor.

In the case of regular graphs, the optimal transport problem can be reduced to an optimal assignment problem between subsets of the spheres. This observation was leveraged in \cite{Hehl} to derive a more concise expression for the Lin--Lu--Yau curvature.

\begin{theorem}[\cite{Hehl}, Theorem 1.1]\label{equal degree}
Let $G$ be a locally finite graph. Let $x$ and $y$ be two adjacent vertices of equal degree $d$. Then the Lin--Lu--Yau curvature is as follows:
$$
\kappa_{\mathrm{LLY}}(x, y) = \frac{1}{d} \left( d + 1 - \min_{\phi: N_x \to N_y} \sum_{v \in N_x} d(v, \phi(v)) \right),
$$
where the minimum is taken over all the bijections $\phi$ from $N_x := S_1(x) \setminus B_1(y)$ to $N_y := S_1(y) \setminus B_1(x)$.
\end{theorem}

We define the lower Lin--Lu--Yau curvature bound of a graph at various scales as follows.

\begin{defi}
Let $G = (V, E)$ be a locally finite graph. For any positive integer $i$, we define the lower Lin--Lu--Yau curvature bound of $G$ at scale $i$ as follows:
$$
\kappa_{\mathrm{LLY}}^{(i)}(G) := \inf_{\substack{x, y \in V \\ d(x, y) = i}} \kappa_{\mathrm{LLY}}(x, y).
$$
\end{defi} 

When $i = 1$, we simply write $\kappa_{\mathrm{LLY}}(G)$ instead of $\kappa_{\mathrm{LLY}}^{(1)}(G)$. The following lemma is a reformulation of \cite[Lemma 2.3]{LLY}, which is proved by the triangle inequality of Wasserstein distance.

\begin{lemma}\label{triangle inequality}
Let $G$ be a locally finite graph. For any positive integer $i$, it holds that
$$
\kappa_{\mathrm{LLY}}^{(i)}(G) \geq \kappa_{\mathrm{LLY}}(G).
$$
\end{lemma}



Hehl established a lower bound on the minimum degree that guarantees non-negative Lin--Lu--Yau curvature. He also proceed to discuss the sharpness of the bound.

\begin{theorem}[\cite{Hehlmin}, Theorem 1.1]\label{Hehlmin}
Let $G$ be a connected graph with $n$ vertices and minimum degree $\delta(G)$. If$$\delta(G)\geq \frac{2n}{3}-1,$$ then $\kappa_{\mathrm{LLY}}(x,y)\geq0$ for all edges $xy$ in $G$.
\end{theorem}

\begin{theorem}[\cite{Hehlmin}, Theorem 1.2]\label{Hehlmin sharpness}
Let $n \in \mathbb{N}$ such that $2n/3 \in 2\mathbb{N}$ and $2n/3 \geq 6$. Then, there exists a graph $G$ with $n$ vertices and minimum degree $$\delta(G) = \frac{2n}{3} - 2,$$ that contains an edge $xy$ with $\kappa_{\mathrm{LLY}}(x, y) < 0$.
\end{theorem}



\section{Positive Curvature Implies Large Edge-Connectivity}\label{section}



In this section, we prove sharp lower bounds of edge-connectivity based on Lin--Lu--Yau curvature. Our results show that a large positive lower bound of Lin--Lu--Yau curvature on each edge (or between any two vertices at distance three) implies large edge-connectivity. Before our proof, we present s basic lemma.

\begin{lemma} \label{neighbor}
Let $G$ be a non-complete connected graph with edge-connectivity $\kappa'(G)$. Let $S \subset E$ be an edge set of size $|S| = \kappa'(G)$ such that $G - S$ is disconnected and $X$, $Y$ be two distinct connected components of $G - S$. Let $X_1:=X\cap V(S),\ X_2:=X \setminus X_1,\ Y_1:=Y\cap V(S),\ Y_2:=Y \setminus Y_1$. Then, for any edge $uw \in S$ with $u\in X_1$ and $w\in Y_1$, if $X_2\ne \emptyset$, $u$ has at least one neighbor in $X_2$; if $Y_2\ne \emptyset$, $w$ has at least one neighbor in $Y_2$.
\end{lemma}

\begin{proof}
Without loss of generality, assume that $X_2\ne\emptyset$. Suppose that $u\in X_1$ has no neighbors in $X_2$. Then $G - (S \setminus \{uw\})$ is disconnected, which is contradictory to the minimality of $S$.
\end{proof}


\begin{proof}[\bf{Proof of Theorem \ref{diam3}}]
Let $G=(V,E)$ be a non-complete connected graph with edge-connectivity $\kappa'(G)$. Let $S \subset E$ be a edge set of size $|S| = \kappa'(G)$ such that $G - S$ is disconnected. Let $X$ and $Y$ be two distinct connected components of $G - S$.
Let us set
$$
X_1:=X\cap V(S),\ X_2:=X \setminus X_1,\ Y_1:=Y\cap V(S),\ Y_2:=Y \setminus Y_1.
$$
Since $\mathrm{diam}(G)\geq3$, at most one of $X_2$ and $Y_2$ is empty. We divide the discussion into two cases according to whether both $X_2$ and $Y_2$ are not empty.

\textbf{Case 1:} One of $X_2$ and $Y_2$ is empty.

Without loss of generality, assume that $X_2\ne\emptyset,\ Y_2=\emptyset$.
Let $u$ be a vertex in $X_1$. Let $x$ be a neighbor of $u$ in $X_2$ (by Lemma \ref{neighbor}, such $x$ must exist). Let $w$ be a neighbor of $u$ in $Y_1$. Then, we have the following:
\begin{align*} 
d(v, X_2) &= 1, \text{ for any } v \in S_1(x) \cap X_1;\\
d(v, X_2) &= 1, \text{ for any } v \in S_1(w) \cap X_1;\\
d(v, X_2) &= 2, \text{ for any } v \in S_1(w) \cap Y_1.
\end{align*}
Define the 1-Lipschitz function $f(v):=d(v, X_2)$. According to Kantorovich duality, we derive
\begin{align*}
W(\mu_w^p, \mu_x^p) &= \sup_{\substack{f: V \to \mathbb{R} \\ 1-\text{Lipschitz}}} \sum_{v \in V} f(v) \left( \mu_w^p(v) - \mu_x^p(v) \right)\\
&\geq \sum_{v \in V} d(v, X_2) \mu_w^p(v) - \sum_{v \in V} d(v, X_2) \mu_x^p(v)\\
&= \frac{|S_1(w) \cap X_1|(1-p)}{d_w} + \left( p + \frac{|S_1(w) \cap Y_1|(1-p)}{d_w} \right) \times 2 - \frac{|S_1(x) \cap X_1|(1-p)}{d_x}\\
&= 2 - \frac{|S_1(w) \cap X_1|(1-p)}{d_w} - \frac{|S_1(x) \cap X_1|(1-p)}{d_x}\\
&\geq 2 - \frac{1-p}{\delta(G)} \left( |S_1(w) \cap X_1| + |S_1(x) \cap X_1| \right).
\end{align*}
In the above equality, we have employed the fact that $d_w = |S_1(w) \cap X_1| + |S_1(w) \cap Y_1|$.

By the definition of Lin--Lu--Yau curvature, we estimate
$$
\kappa_{\mathrm{LLY}}(w, x) = \lim_{p \to 1} \frac{1}{1-p} \left( 1 - \frac{W(\mu_w^p, \mu_x^p)}{d(x, w)} \right) \leq \frac{|S_1(w) \cap X_1| + |S_1(x) \cap X_1|}{2\delta(G)}.
$$
Therefore, we derive
\begin{align*}
\kappa'(G) &= |S| \geq \max\{|S_1(x) \cap X_1|, |S_1(w) \cap X_1|\}\\
&\geq \frac{1}{2} (|S_1(x) \cap X_1| + |S_1(w) \cap X_1|) \geq \delta(G) \kappa_{\mathrm{LLY}}^{(2)}(G).
\end{align*}
We complete the proof of Case 1.

\textbf{Case 2:} Both $X_2$ and $Y_2$ are not empty.

Let $uw \in S$ with $u\in X_1$ and $w\in Y_1$,  let $x$ be a neighbor of $u$ in $X_2$ and $y$ be a neighbor of $w$ in $Y_2$ (by Lemma \ref{neighbor}, such $x$ and $y$ must exist). Then, we have the following:
\begin{align*} 
d(v, X_2) &= 1, \text{ for any } v \in S_1(x) \cap X_1;\\
d(v, X_2) &= 2, \text{ for any } v \in S_1(y) \cap Y_1;\\
d(v, X_2) &\geq 3, \text{ for any } v \in S_1(y) \cap Y_2.
\end{align*}
Define the 1-Lipschitz function $f(v):=d(v, X_2)$. According to Kantorovich duality, we derive
\begin{align*}
W(\mu_y^p, \mu_x^p) &= \sup_{\substack{f: V \to \mathbb{R} \\ 1-\text{Lipschitz}}} \sum_{v \in V} f(v) \left( \mu_y^p(v) - \mu_x^p(v) \right)\\
&\geq \sum_{v \in V} d(v, X_2) \mu_y^p(v) - \sum_{v \in V} d(v, X_2) \mu_x^p(v)\\
&\geq \frac{|S_1(y) \cap Y_1|(1-p)}{d_y}\times 2 + \left( p + \frac{|S_1(y) \cap Y_2|(1-p)}{d_y} \right) \times 3 - \frac{|S_1(x) \cap X_1|(1-p)}{d_x}\\
&= \frac{|S_1(y) \cap Y_1|(2-2p)}{d_y} +  \frac{3p|S_1(y) \cap Y_1|+3|S_1(y) \cap Y_2|}{d_y} - \frac{|S_1(x) \cap X_1|(1-p)}{d_x}\\
&= \frac{3|S_1(y) \cap Y_1|+3|S_1(y) \cap Y_2|}{d_y} +  \frac{|S_1(y) \cap Y_1|(p-1)}{d_y} - \frac{|S_1(x) \cap X_1|(1-p)}{d_x}\\
&= 3 - \frac{|S_1(y) \cap Y_1|(1-p)}{d_y} - \frac{|S_1(x) \cap X_1|(1-p)}{d_x}\\
&\geq 3 - \frac{1-p}{\delta(G)} \left( |S_1(y) \cap Y_1| + |S_1(x) \cap X_1| \right).
\end{align*}
In the above equality, we have employed the fact that $d_y = |S_1(y) \cap Y_1| + |S_1(y) \cap Y_2|$.

By the definition of Lin--Lu--Yau curvature, we estimate
$$
\kappa_{\mathrm{LLY}}(y, x) = \lim_{p \to 1} \frac{1}{1-p} \left( 1 - \frac{W(\mu_y^p, \mu_x^p)}{d(x, y)} \right) \leq \frac{|S_1(y) \cap Y_1| + |S_1(x) \cap X_1|}{3\delta(G)}.
$$
Therefore, we derive
\begin{align*}
\kappa'(G) &= |S| \geq \max\{|S_1(x) \cap X_1|, |S_1(y) \cap Y_1|\}\\
&\geq \frac{1}{2} (|S_1(x) \cap X_1| + |S_1(y) \cap Y_1|) = \frac{3}{2} \times \frac{|S_1(x) \cap X_1| + |S_1(y) \cap Y_1|}{3}\\
&\geq \frac{3}{2}\delta(G) \kappa_{\mathrm{LLY}}^{(3)}(G).
\end{align*}
We complete the proof of Theorem \ref{diam3}.
\end{proof}

The next lemma characterized the structure of the infinite family of graphs for which Theorem \ref{diam3} is sharp.

\begin{lemma} \label{distance3}
Let $G_{2n}\cong K_n \cup K_n +M_2$ be the union of two copies of the complete graph $K_n$ ($n\geq3$) with a matching of size 2 between them. 
For any two vertices $x,y\in V(G_{2n})$ satisfying $d(x,y)=3$, we have
$\kappa_{\mathrm{LLY}}(x,y)=4/3(n-1).$
In particular, $\kappa_{\mathrm{LLY}}^{(3)}(G_{2n})=4/3(n-1).$
\end{lemma}

\begin{proof}
Without loss of generality, write $$V(G_{2n})=A\cup B=\{x,a_1,a_2,u_1,\ldots,u_{n-3}\}\cup \{y,b_1,b_2,v_1,\ldots,v_{n-3}\},$$ where the two edges of the matching are $a_1b_1$ and $a_2b_2$. 
Note that $d(x,y)=3$, $d(a_i,b_i)=1$ for $i=1,2$
and $d(u_i,v_j)=3$ for $i=1,\ldots,n-3$.
Since $d_x=d_y=n-1$, setting $q:=(1-p)/(n-1)$.

Consider the transport plan $\pi$ from $\mu_x^p$ to $\mu_y^p$ given by
\begin{align*} 
\pi(x,y)&=p,\\
\pi(a_1,b_1)&=\pi(a_2,b_2)=q,\\
\pi(u_i,v_i)&=q, \text{ for } i=1,\ldots,n-3.
\end{align*}
Its transport cost is
$$
\operatorname{Cost}(\pi)=3p+2q+3(n-3)q=3-\frac{4}{n-1}(1-p).
$$
Therefore
$$
W(\mu_x^p,\mu_y^p)\leq3-\frac{4}{n-1}(1-p).
$$

To prove optimality, define the 1-Lipschitz function $f:V(G_{2n})\to\mathbb{R}$ by 
$$
f(z)=
\begin{cases}
2, & z\in A\setminus\{a_1,a_2\};\\
1, & z\in\{a_1,a_2\};\\
0, & z\in\{b_1,b_2\};\\
-1, & z\in B\setminus\{b_1,b_2\}.
\end{cases}
$$
According to Kantorovich duality, we derive
\begin{align*}
W(\mu_x^p, \mu_y^p) &= \sup_{\substack{f: V(G_{2n}) \to \mathbb{R} \\ 1-\text{Lipschitz}}} \sum_{z \in V(G_{2n})} f(z) \left( \mu_x^p(z) - \mu_y^p(z) \right)\\
&\geq\sum_{z\in V(G_{2n})}f(z)\mu_x^p(z)
-\sum_{z\in V(G_{2n})}f(z)\mu_y^p(z)\\
&=\left(2p+2q+2(n-3)q\right)
-\left(-p-(n-3)q\right)\\
&=3p+(3n-7)q\\
&=3-\frac{4}{n-1}(1-p).
\end{align*}
Combining with the previous upper bound gives
$$
W(\mu_x^p,\mu_y^p)=3-\frac{4}{n-1}(1-p).
$$
By the definition of Lin--Lu--Yau curvature, we have
$$
\kappa_{\mathrm{LLY}}(x,y)=\lim_{p\to1}\frac{1}{1-p} \left( 1 - \frac{W(\mu_x^p, \mu_y^p)}{d(x, y)} \right)=\frac{4}{3(n-1)}.
$$
Since every pair of vertices at distance three in $G_{2n}$
belongs to the same automorphism orbit, it follows that
$$
\kappa_{\mathrm{LLY}}^{(3)}(G_{2n})=\frac{4}{3(n-1)}.
$$
We complete the proof of Lemma \ref{distance3}.
\end{proof}

\section{Large Edge-Connectivity Implies Large Curvature}

In this section, we give two sharp lower bounds on the edge-connectivity of a graph that guarantee positive Lin--Lu--Yau curvature. These results show that large edge-connectivity can guarantee large Lin--Lu--Yau curvature on graph, which complements the conclusion of Section \ref{section}. Before presenting our proof, we require the following lemma, which guarantees that the diameter of a graph is at most two when the minimum degree is sufficiently large.

\begin{lemma}\label{diam}  
Let $G$ be a graph on $n$ vertices. If the minimum degree $\delta(G)$ of $G$ satisfies
$$
\delta(G) \geq \frac{n-1}{2},
$$
then the diameter of $G$ satisfies $\mathrm{diam}(G)\leq 2$.
\end{lemma}

\begin{proof}
Assume that $\operatorname{diam}(G) > 2$. Hence, there exist vertices $x, y \in V$ such that $d(x, y) > 2$, and therefore $S_1(x) \cap S_1(y) = \emptyset$. As $x, y \notin S_1(x) \cup S_1(y)$, we have
$$
|S_1(x) \cup S_1(y)| \leq n - 2.
$$
Nevertheless, we have
$$
|S_1(x) \cup S_1(y)| = d_x + d_y \geq 2\delta(G) \geq n - 1.
$$
Therefore, our assumption was incorrect, and it follows that $\mathrm{diam}(G) \leq 2$ must hold.
\end{proof}

We also require the following lemma, which determines $\kappa'(G)$ when the minimum degree is sufficiently large.

\begin{lemma}\label{kappadelta}
If a graph satisfies $\delta(G)\geq (n-1)/2$, then
$\kappa'(G)=\delta(G)$.
\end{lemma}

\begin{proof}
Let $S$ be a non-empty proper vertex subset,
and let $s:=|S|\leq n/2$.
Every vertex of $S$ has at most $s-1$ neighbors inside $S$, hence at least $(\delta(G)-s+1)$ neighbors outside $S$. Therefore, 
$$
E(S,V-S)\geq s(\delta(G)-s+1).
$$
Since $s\le\delta(G)$,
$$
s(\delta(G)-s+1)\ge \delta(G).
$$

Thus every edge-cut has size at least $\delta(G)$.
On the other hand, removing all edges incident with a
minimum-degree vertex disconnects the graph and uses exactly $\delta(G)$ edges. Hence 
$$\kappa'(G)=\delta(G).$$
We complete the proof of Lemma \ref{kappadelta}.
\end{proof}


\begin{proof}[\bf{Proof of Theorem \ref{n bound}}]
Let $xy$ be an arbitrary edge in $G$ and assume, without loss of generality, that $d_x \leq d_y$. Choose
$$\alpha = \frac{1}{d_x + 1}.$$
Let $\pi$ be an optimal transport plan, transporting $\mu_x^\alpha$ to $\mu_y^\alpha$. Denote by $\nu_i^\alpha$ the mass transported with distance $i$ under $\pi$. According to Lemma \ref{diam} and the fact that $\delta(G)\geq \kappa'(G)$, we have $\nu_3^\alpha = 0$ and therefore
$$
\nu_0^\alpha + \nu_1^\alpha + \nu_2^\alpha = 1.
$$
Thus, we obtain
\begin{align*} 
\kappa_\alpha(x, y) &= 1 - W_1(\mu_x^\alpha, \mu_y^\alpha)\\
&= \nu_0^\alpha + \nu_1^\alpha + \nu_2^\alpha - \nu_1^\alpha - 2\nu_2^\alpha\\
&= \nu_0^\alpha - \nu_2^\alpha.
\end{align*}

According to Lemma \ref{simple transport plan}, we may assume
$$
\pi(z, z) = \min\{\mu_x^\alpha(z), \mu_y^\alpha(z)\}
$$
for every $z \in V$. Therefore, $\pi(z, z) = (1 - \alpha)/d_y$ for every $z \in N_{xy}$, $\pi(x, x) = (1 - \alpha)/d_y$, and $\pi(y, y) = (1 - \alpha)/d_x$. This yields
$$\nu_0^\alpha = \frac{1 - \alpha}{d_y} (|N_{xy}| + 1) + \frac{1 - \alpha}{d_x}.$$
Since $\nu_2^\alpha \leq 1 - \nu_0^\alpha$, we conclude
\begin{align*} 
\kappa_\alpha(x, y) &\geq 2\nu_0^\alpha - 1\\
&= 2(1 - \alpha) \left[ \frac{|N_{xy}| + 1}{d_y} + \frac{1}{d_x} - \frac{1}{2(1 - \alpha)} \right]\\
&= 2(1 - \alpha) \left[ \frac{|N_{xy}| + 1}{d_y} + \frac{1}{d_x} - \frac{d_x + 1}{2d_x} \right]\\
&= 2(1 - \alpha) \left[ \frac{|N_{xy}| + 1}{d_y} + \frac{1}{2d_x} - \frac{1}{2} \right]\\
&\geq 2(1 - \alpha) \left[ \frac{2|N_{xy}| + 2}{2d_y} + \frac{1}{2d_y} - \frac{1}{2} \right]\\
&= (1 - \alpha) \left[ \frac{2|N_{xy}| + 3}{d_y} - 1 \right].
\end{align*}
According to Theorem \ref{relation}, the equality $\kappa_\alpha(x, y) = (1-\alpha)\kappa_{\mathrm{LLY}}(x, y)$ holds. Hence, 
$$
\kappa_{\mathrm{LLY}}(x, y) \geq \frac{2|N_{xy}| + 3-d_y}{d_y}.
$$
We next prove that the right-hand side of the inequality is strictly positive.

Since $d_x,d_y\ge \delta(G)$, we have 
$$
|N_{xy}|\geq d_x+d_y-n\geq 2\delta(G)-n\ge 2\kappa'(G)-n.
$$
Furthermore, 
$$
d_y\leq n-d_x+|N_{xy}|\le n-\kappa'(G)+|N_{xy}|.
$$
Putting the inequalities together yields
\begin{align*}
2|N_{xy}|+3-d_y&\geq |N_{xy}|+\kappa'(G)+3-n\\
&\geq(2\kappa'(G)-n)+\kappa'(G)+3-n\\
&=3\kappa'(G)+3-2n.
\end{align*}

When $\kappa'(G)>2n/3-1$, then $2|N_{xy}|+3-d_y\geq3\kappa'(G)+3-2n>0$, yielding $\kappa_{\mathrm{LLY}}(x,y)>0$.

When $\kappa'(G)=2n/3-1$, since $\kappa'(G)$ is an integer, we may assume that 
$$
n=3r,\qquad\kappa'(G)=2r-1.
$$
Suppose that equality $2|N_{xy}|+3-d_y=3\kappa'(G)+3-2n$ holds, so that $2|N_{xy}|+3-d_y=0$ (If not, $2|N_{xy}|+3-d_y>0$ yields $\kappa_{\mathrm{LLY}}(x,y)>0$).
Then every inequality above is forced to be equality, we have $|N_{xy}|=2\kappa'(G)-n$ and $d_y=n-\kappa'(G)+|N_{xy}|.$ Hence, 
$$
d_x=d_y=\kappa'(G)=2r-1,\qquad|N_{xy}|=r-2.
$$
Define $N_x:=S_1(x)\setminus B_1(y)$ and $N_y:=S_1(y)\setminus B_1(x)$,
then $|N_x|=|N_y|=r$.
We claim that there must exist an edge between $N_x$ and $N_y$.
Indeed, suppose not. Choose $u\in N_x$, since 
$$|\{x,y\}|+|N_X|+|N_Y|+|N_{xy}|=2+r+r+r-2=3r=n,$$ 
then all neighbors of $u$ lie in $\{x\}\cup(N_x\setminus\{u\})\cup N_{xy}$, and therefore
$$
d_u\leq1+(r-1)+(r-2)=2r-2<2r-1=\delta(G),
$$
contradicting the definition of $\delta(G)$.
Hence there exists $uv\in E(G)$ with $u\in N_x$ and $v\in N_y$.
Therefore
$$\min_{\phi: N_x \to N_y} \sum_{v \in N_x} d(v, \phi(v))\leq1+2(r-1)=2r-1.$$
Since $d_x=d_y=2r-1$, according to Theorem \ref{equal degree}, 
$$\kappa_{\mathrm{LLY}}(x,y)\ge\frac1{2r-1}\Bigl((2r-1)+1-(2r-1)\Bigr)=\frac1{2r-1}>0.$$

Since the edge $xy$ was chosen arbitrarily, every edge has positive curvature. We complete the proof of Theorem \ref{n bound}.
\end{proof}

Next, we discuss the sharpness of the lower bound.


\begin{proof}[\bf{Proof of Theorem \ref{n bound sharpness}}]
Let $l\geq2$ and define the vertex set
$$
V =\{x,y,v\}\sqcup Z \sqcup X \sqcup Y = \{x, y, v, z_1, \dots, z_{l-2}, x_1, \dots, x_{l+1}, y_1, \dots, y_{l+1}\},
$$
where $|Z|=l-2$, $|X|=|Y|=l+1$.
The graph is defined as follows:
\begin{itemize}
\item $xy\in E$;
\item $x$ is adjacent to all vertices of $Z\cup X$;
\item $y$ is adjacent to all vertices of $Z\cup Y$;
\item every vertex of $Z$ is adjacent to all vertices of $X\cup Y$;
\item $X$ induces a clique, $Y$ induces a clique;
\item $v$ is adjacent to all vertices of $X\cup Y$.
\end{itemize}
Note that there are no vertices $z_i$ in the case where $l = 2$. The vertices $x$ and $y$ are of degree $2l$ and the vertex $v$ is of degree $2l+2$.
Additionally, we have
\begin{align*}
d_{x_i} &= 1 + (l - 2) + l  + 1 = 2l,\quad \forall i \in \{1, \dots, l+1\};\\
d_{y_i} &= 1 + (l - 2) + l  + 1 = 2l,\quad \forall i \in \{1, \dots, l+1\};\\
d_{z_i} &= 2 + 2(l+1) = 2l + 4,\quad \forall i \in \{1, \dots, l-2\}.
\end{align*}
Therefore, the resulting graph satisfies $\delta(G) = 2l$, since $n = 3l + 3$ we have
$$
\delta(G) = \frac{2n}{3} - 2.
$$

Observe that
$$N_x:=S_1(x) \setminus B_1(y) = \{x_1, \dots, x_{l+1}\},\
N_y:=S_1(y) \setminus B_1(x) = \{y_1, \dots, y_{l+1}\}.$$
Since $d_x = d_y$, we can apply Theorem \ref{equal degree}. By construction, the equality $d(x_i, y_j) = 2$ holds for every $i, j \in \{1, \dots, l+1\}$. Hence, $d(z, \phi(z)) = 2$ for every $z \in S_1(x) \setminus B_1(y)$ and every bijection $\phi : S_1(x) \setminus B_1(y) \to S_1(y) \setminus B_1(x)$. Thus, we conclude
\begin{align*}
\kappa_{\mathrm{LLY}}(x, y) &= \frac{1}{d_x} \left( d_x + 1 - \inf_{\phi} \sum_{z \in S_1(x) \setminus B_1(y)} d(z, \phi(z)) \right)\\
&= \frac{1}{d_x} \left( d_x + 1 - 2|S_1(x) \setminus B_1(y)| \right)\\
&= \frac{1}{2l} \left( 2l + 1 - 2 \left( l + 1 \right) \right)\\
&=  -\frac{1}{2l}.
\end{align*}

Note that $\delta(G)=2l\geq(3l+2)/2=(n-1)/2.$
Applying Lemma \ref{kappadelta} to the above construction gives
$\kappa'(G)=\delta(G)=2l=2n/3-2.$
Therefore, there exist infinitely many graphs with $\kappa'(G)=2n/3-2$ that contain an edge $xy$ with $\kappa_{\mathrm{LLY}}(x, y) < 0$.
We complete the proof of Theorem \ref{n bound sharpness}.
\end{proof}

Combining Lemma \ref{kappadelta} with Theorem \ref{n bound} yields the following corollary, which can be considered a strong extension of Theorem \ref{Hehlmin}.

\begin{cor}
Let $G$ be a connected graph with $n$ vertices and minimum degree $\delta(G)$. If$$\delta(G)\geq \frac{2n}{3}-1,$$ then $\kappa_{\mathrm{LLY}}(x,y)>0$ for all edges $xy$ in $G$.    
\end{cor}

We also derive a sufficient condition, depending only on the number of edges, for positive Lin--Lu--Yau curvature.


\begin{proof}[\bf{Proof of Theorem \ref{m bound}}]
Let $\kappa:=\kappa'(G)$, since $\kappa\leq \delta(G)$, every vertex has degree at least $\kappa$.
Hence $2m=\sum_{v\in V(G)} d_v \geq n\kappa$,
therefore $n\leq 2m/\kappa$.
By the condition $\kappa\geq 2n/3-1$ of Theorem \ref{n bound}, it suffices to prove that
$$\kappa\geq \frac{4m}{3\kappa}-1.$$
After rearranging, $\kappa^2+\kappa-4m/3\geq 0.$
The positive root of $x^2+x-4m/3=0$ is
$$x=\frac{\sqrt{1+\frac{16}{3}m}-1}{2}.$$
Hence condition
$$
\kappa\geq \left\lceil
\frac{\sqrt{1+\frac{16}{3}m}-1}{2}
\right\rceil
$$
implies $\kappa^2+\kappa-4m/3\geq 0$, and therefore $\kappa\geq 2n/3-1$.
Applying Theorem\ref{n bound} yields $\kappa_{\mathrm{LLY}}(x,y)>0$
for all edges $xy$ in $G$. We complete the proof of Theorem \ref{m bound}. 
\end{proof}

\begin{remark}
Asymptotically, the condition of
Theorem \ref{m bound} can be written as
$$
\kappa'(G)\geq \frac{2}{\sqrt3}\sqrt m + O(1).
$$
Thus an edge-connectivity of order $\sqrt m$
already forces positive Lin--Lu--Yau curvature.
\end{remark}

Next, we discuss the sharpness of the lower bound.


\begin{proof}[\bf{Proof of Theorem \ref{m bound sharpness}}]
Let $t\geq0$ and define the vertex set
$$
V =\{x,y,v\}\sqcup A \sqcup X \sqcup Y = \{x, y, v, a_1, \dots, a_t, x_1, \dots, x_{t+2}, y_1, \dots, y_{t+2}\},
$$
where $|A|=t$, $|X|=|Y|=t+2$.
The graph is defined as follows:
\begin{itemize}
\item $xy\in E$;
\item $x$ is adjacent to all vertices of $A\cup X$;
\item $y$ is adjacent to all vertices of $A\cup Y$;
\item every vertex of $A$ is adjacent to all vertices of $X\cup Y$;
\item $X$ induces a clique, $Y$ induces a clique;
\item $v$ is adjacent to all vertices of $X\cup Y$.
\end{itemize}
Note that there are no vertices $a_i$ in the case where $t = 0$. The vertices $x$ and $y$ are of degree $2t+3$ and the vertex $v$ is of degree $2t+4$.
Additionally, we have
\begin{align*}
d_{x_i} &= 1 + t + (t+1) + 1= 2t+3,\quad \forall i \in \{1, \dots, t+2\};\\
d_{y_i} &= 1 + t + (t+1) + 1= 2t+3,\quad \forall i \in \{1, \dots, t+2\};\\
d_{a_i} &= 2 + 2(t+2)= 2t + 6,\quad \forall i \in \{1, \dots, t\}.
\end{align*}
Therefore, the resulting graph satisfies $\delta(G) = 2t+3$, since $n = 3t + 7$ we have
$$
\delta(G) =2t+3\geq \frac{3t+6}{2} =\frac{n-1}{2}.
$$
By Lemma \ref{kappadelta}, we have
$$\kappa'(G)=\delta(G)=2t+3.$$

Observe that
$$
N_x:=S_1(x) \setminus B_1(y) = \{x_1, \dots, x_{t+2}\},\
N_y:=S_1(y) \setminus B_1(x) = \{y_1, \dots, y_{t+2}\}.
$$
Since $d_x = d_y$, we can apply Theorem \ref{equal degree}. By construction, the equality $d(x_i, y_j) = 2$ holds for every $i, j \in \{1, \dots, l+1\}$. Hence, $d(z, \phi(z)) = 2$ for every $z \in S_1(x) \setminus B_1(y)$ and every bijection $\phi : S_1(x) \setminus B_1(y) \to S_1(y) \setminus B_1(x)$. Thus, we conclude
\begin{align*}
\kappa_{\mathrm{LLY}}(x, y) &= \frac{1}{d_x} \left( d_x + 1 - \inf_{\phi} \sum_{z \in S_1(x) \setminus B_1(y)} d(z, \phi(z)) \right)\\
&= \frac{1}{d_x} \left( d_x + 1 - 2|S_1(x) \setminus B_1(y)| \right)\\
&= \frac{1}{2t + 3} \left( 2t + 3+1 - 2 (t+2) \right)\\
&= 0.
\end{align*}
As for the number of edges, we have
\begin{align*}
m &= 1 + 2(2t+2) + 2t(t+2) + 2\binom{t+2}{2} + 2(t+2)\\
&= 1 + 4t+4 + 2t^2+4t + t^2+3t+2 + 2t+4\\
&=3t^2+13t+11.
\end{align*}
Thus, we conclude
$$\kappa'(G)=2t+3=(2t+4)-1
=\left\lceil\frac{\sqrt{1+\frac{16}{3}m}-1}{2}\right\rceil-1.$$ 

Therefore there exist infinitely many graphs with $\kappa'(G)=\left\lceil\frac{\sqrt{1+\frac{16}{3}m}-1}{2}\right\rceil-1$ that contain an edge $xy$ with $\kappa_{\mathrm{LLY}}(x, y) = 0$.
We complete the proof of Theorem \ref{m bound sharpness}.
\end{proof}


\end{document}